\documentclass[reqno,11pt,twoside]{article}
\usepackage[hmargin=3cm, vmargin=1.in, marginparwidth=2.6cm, marginparsep=0.3cm, a4paper, centering]{geometry}
\usepackage[T1]{fontenc}
\usepackage[lining]{ebgaramond}
\usepackage{amsmath,amsthm}
\usepackage[ebgaramond]{newtxmath}
\ifdefined\iflatexml
\theoremstyle{plain}
\newtheorem{theorem}{Theorem}
\newtheorem{prop}[theorem]{Proposition}
\newtheorem{lemma}[theorem]{Lemma}

\theoremstyle{definition}

\newtheorem{remark}[theorem]{Remark}
\else
\usepackage{tcolorbox,varwidth}
\tcbuselibrary{skins, breakable, theorems}
\colorlet{thmcolor}{orange!80!white}
\colorlet{remcolor}{teal!50!white}
\colorlet{defncolor}{blue!50!black}
\usepackage{keytheorems}
\tcbset{commonstyle/.style={%
colback=white!95!tcbcolframe,
colbacktitle=white!80!tcbcolframe,
arc=2mm,
fonttitle=\bfseries,
coltitle=black,
enhanced,
varwidth boxed title*=-2cm,
attach boxed title to top left={yshift=-3mm,xshift=0.5cm,yshifttext=-1mm},
description font=\mdseries,
breakable,
before upper={\parindent15pt\noindent},
beforeafter skip balanced=10.0pt plus 1.0pt minus 1.0pt,
separator sign={\textbf{: }},
overlay first={\draw[color=tcbcolframe, line width=.5pt] (frame.south west)--(frame.south east);},
overlay middle={\draw[color=tcbcolframe, line width=.5pt] (frame.south west)--(frame.south east);\draw[color=tcbcolframe, line width=.5pt] (frame.north west)--(frame.north east);},
overlay last={\draw[color=tcbcolframe, line width=.5pt] (frame.north west)--(frame.north east);}
}
}
\newkeytheoremstyle{mythmstyle}{
bodyfont=\normalfont,
headpunct={},
notebraces={}{},
noteseparator={: }, 
notefont=\bfseries,
tcolorbox = {colframe=thmcolor,commonstyle}
}
\newkeytheoremstyle{myremstyle}{
bodyfont=\normalfont,
headpunct={},
notebraces={}{},
noteseparator={: }, 
notefont=\bfseries,
tcolorbox = {colframe=remcolor,commonstyle}
}
\newkeytheoremstyle{mydefstyle}{
bodyfont=\normalfont,
headpunct={},
notebraces={}{},
noteseparator={: }, 
notefont=\bfseries,
tcolorbox = {colframe=defncolor,commonstyle}
}
\newkeytheorem{theorem}[
name=Theorem,
style=mythmstyle
]
\newkeytheorem{lemma}[
name=Lemma,
style=mythmstyle,
sibling=theorem
]
\newkeytheorem{prop}[
name=Proposition,
style=mythmstyle,
sibling=theorem
]
\newkeytheorem{corollary}[
name=Corollary,
style=mythmstyle,
sibling=theorem
]
\newkeytheorem{remark}[
name=Remark,
sibling=theorem,
style=myremstyle
]
\newkeytheorem{definition}[
name=Definition,
style=mydefstyle,
sibling=theorem
]
\fi
\usepackage{mleftright} 
\mleftright 
\usepackage[nobottomtitles,pagestyles]{titlesec}
\titleformat{\section}
{\normalfont\large\bfseries}
{\filcenter\S\thesection.}{1ex}{\filcenter}
\titleformat{\subsection}
{\normalfont\bfseries}
{\filcenter\S\thesubsection.}{1ex}{\filcenter}
\usepackage[colorlinks,allcolors=blue,pagebackref]{hyperref}
\usepackage{pifont}
\renewcommand*{\backrefalt}[4]{%
\ifcase #1 %
\textcolor{red}{No citations}%
\or
\ding{43}~p.~#2%
\else
\ding{43}~pp.~#2%
\fi}
\usepackage{breakcites}
\newcommand{\mydoi}[1]{\href{https://doi.org/#1}{doi: #1}}
\newcommand{\myarXiv}[1]{\href{https://arxiv.org/abs/#1}{arXiv: #1}}
\newcommand{\Om}{\Omega}
\newcommand{\R}{\mathbb R}
\newcommand{\C}{\mathbb C}
\newcommand{\grad}{\nabla}
\newcommand{\ir}{\mathrm{i}}
\newcommand{\dr}{\mathrm{d}}
\renewcommand\footnotemark{}
\begin{document}
\title{%
Payne's conjecture for buckling eigenvalues of odd index%
}
\author{
Michael Levitin\hspace{-3ex}
\thanks{%
\textbf{M. L.: }Department of Mathematics and Statistics, University of Reading, 
Pepper Lane, Whiteknights, Reading RG6 6AX, UK;
\href{mailto:M.Levitin@reading.ac.uk}{\nolinkurl{M.Levitin@reading.ac.uk}}; \url{https://www.michaellevitin.net}; ORCID: 0000-0003-0020-3265%
}
\and 
Iosif Polterovich
\thanks{%
\textbf{I. P.: }D\'e\-par\-te\-ment de math\'ematiques et de statistique, Univer\-sit\'e de Mont\-r\'eal, 
CP 6128 succ Centre-Ville, Mont\-r\'eal QC  H3C 3J7, Canada;
\href{mailto:iosif.polterovich@umontreal.ca}{\nolinkurl{iosif.polterovich@umontreal.ca}}; \url{https://www.dms.umontreal.ca/\~iossif}; ORCID: 0009-0007-0052-6589%
}
}
\date{\small arXiv:2609.12235v2, 28 September 2026} 
\maketitle

\begin{abstract}
In 1955, L. E. Payne conjectured that for a planar clamped plate, each buckling eigenvalue is at least as large as the Dirichlet Laplacian eigenvalue of the same domain with index shifted by one. We prove this conjecture for every odd buckling index.
\end{abstract}

\section{Introduction and main results}

Let $\Om\subset\R^2$ be a bounded  domain. We consider the Dirichlet eigenvalue problem for the Laplacian
\begin{equation}\label{eq:dirichlet-pde}
-\Delta\varphi=\lambda\varphi\quad\text{in }\Om,
\qquad
\varphi=0\quad\text{on }\partial\Om,
\end{equation}
and the buckling eigenvalue problem for the clamped plate
\begin{equation}\label{eq:buckling-pde}
\Delta^2u+\Lambda\Delta u=0\quad\text{in }\Om,
\qquad
u=\partial_\nu u=0\quad\text{on }\partial\Om.
\end{equation}
Both problems are understood in the weak sense, with $\varphi\in H_0^1(\Om)$ and $u\in H_0^2(\Om)$; for a Lipschitz domain, $H_0^2(\Om)$ consists  of those $u\in H^2(\Om)$ for which the traces of $u$ and $\partial_\nu u$ vanish on $\partial\Om$. That is, $0\neq u\in H_0^2(\Om)$ is a buckling eigenfunction with eigenvalue $\Lambda$ if and only if
\begin{equation}\label{eq:buckling-weak}
\int_\Om\Delta u\,\Delta\overline v=\Lambda\int_\Om\grad u\cdot\grad\overline v
\qquad\text{for all }v\in H_0^2(\Om).
\end{equation}
We denote by
\[
0<\lambda_1\leq\lambda_2\leq\cdots
\qquad\text{and}\qquad
0<\Lambda_1\leq\Lambda_2\leq\cdots
\]
the Dirichlet and buckling eigenvalues, respectively, repeated according to multiplicity. They are given by the variational principles
\begin{equation}\label{eq:dirichlet-minmax}
\lambda_k
=
\min_{\substack{L\subset H_0^1(\Om)\\ \dim L=k}}\ \max_{0\neq \varphi\in L}
\frac{\displaystyle\int_\Om |\grad \varphi|^2}{\displaystyle\int_\Om |\varphi|^2}
\end{equation}
and
\begin{equation}\label{eq:buckling-minmax}
\Lambda_k
=
\min_{\substack{L\subset H_0^2(\Om)\\ \dim L=k}}\ \max_{0\neq u\in L}
\frac{\displaystyle\int_\Om |\Delta u|^2}{\displaystyle\int_\Om |\grad u|^2}.
\end{equation}

The right-hand sides of \eqref{eq:dirichlet-minmax} and \eqref{eq:buckling-minmax} make sense for an arbitrary bounded open set $\Om\subset\R^2$, and from now on we use them as the \emph{definitions} of $\lambda_k(\Om)$ and $\Lambda_k(\Om)$, with no assumption on $\partial\Om$. Indeed, for any bounded open set, 
the embeddings $H_0^1(\Om)\hookrightarrow L^2(\Om)$ and $H_0^2(\Om)\hookrightarrow H_0^1(\Om)$ are compact, 
hence both spectra are discrete, and the minimum in \eqref{eq:buckling-minmax} is attained on the span of the first $k$ buckling eigenfunctions, understood in the sense of \eqref{eq:buckling-weak}.

The following inequality between the two spectra is usually referred to as \emph{Payne's conjecture}  \cite{Payne1955}:
\begin{equation}\label{eq:payne-conj}
\lambda_{k+1}\leq\Lambda_k,
\qquad k\geq1.
\end{equation}
For $k=1$, that is, $\lambda_2\leq\Lambda_1$, it was proved by Payne in the same paper. His original argument contained a gap, later repaired by Friedlander \cite{Friedlander2004}. Friedlander also showed that the stronger inequality $\lambda_3\leq\Lambda_1$ fails in general for planar domains, while it does hold under an additional fourfold rotational symmetry. Thus, at the first level, \eqref{eq:payne-conj} is the natural sharp comparison.

For general $k$, the same-index inequality $\lambda_k<\Lambda_k$ for bounded planar domains with locally Lipschitz boundary was obtained by Kelliher \cite{Kelliher2010}, 
cf.\  Remark~\ref{rem:same-index}. 
Friedlander \cite{Friedlander2021} discussed a reformulation of \eqref{eq:payne-conj} in terms of the negative eigenvalues of a Neumann-to-Laplacian boundary operator, pointed out that the corresponding variational problem cannot be posed on the whole of $H^2(\Om)\cap H_0^1(\Om)$, and suggested that a proof of \eqref{eq:payne-conj} should use features specific to Euclidean space.
More recently, Buoso and Freitas \cite{BuosoFreitas2025} revisited Payne's inequality in a general polyharmonic framework.

Our main result establishes \eqref{eq:payne-conj} for every odd $k$.

\begin{theorem}\label{thm:main}
Let $\Om\subset\R^2$ be a bounded open set. Then
\begin{equation}\label{eq:main}
\lambda_{2m}(\Om)\leq\Lambda_{2m-1}(\Om),
\qquad m\geq1.
\end{equation}
Equivalently, Payne's conjecture holds for every odd buckling index.
\end{theorem}

The new point is a first-order way of producing the additional Dirichlet trial direction required by the shift in \eqref{eq:payne-conj}. Fix $k$, let $E_k$ be the span of the first $k$ buckling eigenfunctions, and let $t>\Lambda_k$. Then both $E_k$ and its image under
\[
T:=\partial_1+\ir \partial_2
\]
are negative subspaces with respect to the shifted Dirichlet form $\int_\Om|\grad f|^2-t\int_\Om|f|^2$. The difficulty is that the sum of two negative subspaces need not be negative, because of the cross terms. We encode these cross terms by a complex bilinear form on $E_k$, which is skew-symmetric by integration by parts. If $k$ is odd, a skew-symmetric form on the $k$-dimensional space $E_k$ is degenerate, and this yields $0\neq v\in E_k$ such that $Tv$ is orthogonal to the whole of $E_k$ with respect to the shifted Dirichlet form. Since $Tv\in TE_k$ is itself negative, $E_k\oplus\C\,Tv$ is a $(k+1)$-dimensional negative subspace, and the theorem follows by min--max.

This also explains the limitation of the argument. For even $k$, a skew-symmetric form may be non-degenerate, and the present construction need not produce any additional direction. The parity is not merely formal: in one dimension the analogous shifted inequality is an equality at odd indices and fails at even indices; see \cite[Section~4]{BLPS2021}. Thus any proof of the remaining even-index cases in the plane must use information beyond the odd-dimensional skew-symmetry exploited here; see Remark~\ref{rem:even} for a precise statement. The argument below does not suggest that Payne's conjecture should fail at even planar indices; it only identifies the point where this particular mechanism stops.

There is a link between Payne's conjecture and comparison inequalities
for Dirichlet and Neumann eigenvalues. In fact, they were originally proposed in the same paper \cite{Payne1955}, see \S\ref{subsec:geometry} for further discussion.
Our approach is motivated by recent developments in this topic,
see e.g. \cite{Rohleder2025} and notably a new preprint \cite{TangZhang2026}, see also our AI usage disclosure statement.

\section{Proof of the main theorem}

\subsection{The variational argument}\label{subsec:planar-proof}

We work with complex-valued Sobolev spaces, and write $\|\cdot\|_2$ for the norm of $L^2(\Om;\C)$. For $t>0$, let 
\[
q_t(f,g)
:=\int_\Om \grad f\cdot\grad\overline g
-t\int_\Om f\overline g
\]
be the shifted Dirichlet form on $H_0^1(\Om;\C)$, and let
\[
b_t(f,g)
:=\int_\Om \grad f\cdot\grad g
-t\int_\Om fg
\]
be its complex bilinear companion. Thus $q_t$ is Hermitian, $b_t$ is symmetric, and $q_t(f,g)=b_t(f,\overline g)$. As usual we write $q_t[f]:=q_t(f,f)$ for the associated quadratic form.

Choose real buckling eigenfunctions $u_1,\ldots,u_k$ such that
\begin{equation}\label{eq:normalization}
\int_\Om\grad u_i\cdot\grad u_j=\delta_{ij},
\qquad
\int_\Om\Delta u_i\,\Delta u_j=\Lambda_i\delta_{ij},
\end{equation}
and set
\[
E_k:=\operatorname{span}_{\C}\{u_1,\ldots,u_k\}.
\]
By the first relation in \eqref{eq:normalization}, the $u_i$ remain linearly independent over $\C$, so $\dim_{\C}E_k=k$; moreover $E_k$ is invariant under complex conjugation. Expanding $u=\sum_i c_iu_i$ and using \eqref{eq:normalization} gives, for every $u\in E_k$,
\begin{equation}\label{eq:buckling-bound}
\|\Delta u\|_2^2=\sum_{i=1}^k\Lambda_i|c_i|^2\leq\Lambda_k\sum_{i=1}^k|c_i|^2=\Lambda_k\|\grad u\|_2^2.
\end{equation}

The proof rests on the following elementary first-order identities.

\begin{lemma}\label{lem:first-order}
The operator $T=\partial_1+\ir\partial_2$ maps $H_0^2(\Om;\C)$ continuously into $H_0^1(\Om;\C)$. For $u,v\in H_0^2(\Om;\C)$,
\begin{align}
\|Tu\|_2^2&=\|\grad u\|_2^2,\label{eq:T1}\\
\|\grad Tu\|_2^2&=\|\Delta u\|_2^2,\label{eq:T2}
\end{align}
and
\begin{equation}\label{eq:skew}
b_t(Tu,v)=-b_t(u,Tv).
\end{equation}
\end{lemma}

\begin{proof}
The mapping property follows directly from the definition of $H_0^2(\Om)$ as the closure of $C_c^\infty(\Om)$ in $H^2(\Om)$. Extend $u$ by zero onto $\R^2$. This extension belongs to $H^2(\R^2)$. 
Identifying $\R^2\simeq\C$ via $\xi=\xi_1+\ir\xi_2$, we have
$\widehat{Tu}(\xi)=\ir\xi\,\widehat u(\xi)$, and $|\ir\xi|=|\xi|$. Plancherel's
theorem then gives
\eqref{eq:T1} and \eqref{eq:T2}. Finally, for compactly supported smooth functions, each $\partial_j$ is formally skew-symmetric with respect to the complex bilinear $L^2$ pairing. Since $T$ commutes with the derivatives, integration by parts gives \eqref{eq:skew}; the general case follows by density.
\end{proof}

We next record the two negative trial spaces. If $0\neq u\in E_k$, then $\grad u\neq0$, and integration by parts, Cauchy--Schwarz and \eqref{eq:buckling-bound} give
\[
\|\grad u\|_2^2
=\left|\int_\Om(-\Delta u)\overline u\right|
\leq\|\Delta u\|_2\|u\|_2
\leq\Lambda_k^{1/2}\|\grad u\|_2\|u\|_2,
\]
so that $\|\grad u\|_2\leq\Lambda_k^{1/2}\|u\|_2$, and hence
\begin{equation}\label{eq:Enegative}
q_t[u]
=\|\grad u\|_2^2-t\|u\|_2^2
\leq(\Lambda_k-t)\|u\|_2^2<0
\end{equation}
whenever $t>\Lambda_k$. On the other hand, Lemma~\ref{lem:first-order} and \eqref{eq:buckling-bound} yield
\begin{equation}\label{eq:TEnegative}
q_t[Tu]
=\|\Delta u\|_2^2-t\|\grad u\|_2^2
\leq(\Lambda_k-t)\|\grad u\|_2^2<0.
\end{equation}
Thus both $E_k$ and $TE_k$ are $q_t$-negative.

Define a complex bilinear form on $E_k$ by
\[
\omega_t(u,v):=b_t(Tu,v).
\]
By \eqref{eq:skew} and the symmetry of $b_t$,
\[
\omega_t(u,v)
=-b_t(u,Tv)
=-b_t(Tv,u)
=-\omega_t(v,u).
\]
Hence $\omega_t$  is skew-symmetric.

\begin{proof}[Proof of Theorem~\ref{thm:main}]
Let $k$ be odd and fix $t>\Lambda_k$. 
Since $E_k$ has odd complex dimension, the skew-symmetric form $\omega_t$ is degenerate
(its Gram matrix $M$ in any basis satisfies $\det M=\det M^{\mathsf T}=\det (-M)=(-1)^k\det M$, so $\det M=0$).
Choose $0\neq v\in E_k$ in its kernel, so that
\begin{equation}\label{eq:radical}
b_t(Tv,w)=0
\qquad\text{for every }w\in E_k.
\end{equation}
Because $E_k$ is invariant under complex conjugation, \eqref{eq:radical} implies
\[
q_t(Tv,w)=b_t(Tv,\overline w)=0
\qquad\text{for every }w\in E_k,
\]
and, $q_t$ being Hermitian, also $q_t(w,Tv)=0$ for every $w\in E_k$. By \eqref{eq:TEnegative}, $q_t[Tv]<0$. In particular $Tv\notin E_k$, since otherwise the preceding orthogonality with $w=Tv$ would give $q_t[Tv]=0$.

It follows from \eqref{eq:Enegative} and \eqref{eq:TEnegative} that $q_t$ is negative definite on the $(k+1)$-dimensional space
\[
E_k\oplus\C Tv.
\]
Indeed, the two summands are $q_t$-orthogonal and $q_t$ is negative on each of them. The min--max principle \eqref{eq:dirichlet-minmax} therefore gives
\[
\lambda_{k+1}<t.
\]
Since this holds for every $t>\Lambda_k$, letting $t\downarrow\Lambda_k$ yields
\[
\lambda_{k+1}\leq\Lambda_k.
\]
Taking $k=2m-1$ proves \eqref{eq:main}.
\end{proof}

\begin{remark}\label{rem:same-index}
The first part of the proof of \eqref{eq:Enegative} alone gives $\lambda_k<\Lambda_k$ for every $k$ and every bounded open set $\Om\subset\R^2$; the argument works verbatim in any dimension. Indeed, we showed that $\|\grad u\|_2^2\leq\Lambda_k\|u\|_2^2$ for all $u\in E_k$. If equality holds for some $u\neq0$, then both inequalities in the preceding chain are equalities. In particular, equality in the Cauchy--Schwarz inequality gives $-\Delta u=c\,u$ for some $c\in\C$, and then $c=\|\grad u\|_2^2/\|u\|_2^2=\Lambda_k$. Since $u\in H_0^2(\Om)$, its extension $\tilde u$ by zero belongs to $H^2(\R^2)$ and satisfies $-\Delta\tilde u=\Lambda_k\tilde u$ on $\R^2$. Hence $(|\xi|^2-\Lambda_k)\,\widehat{\tilde u}(\xi)=0$ for almost every $\xi$, and so $\tilde u=0$, a contradiction. By compactness of the unit sphere of $E_k$,
\[
\max_{0\neq u\in E_k}\frac{\|\grad u\|_2^2}{\|u\|_2^2}<\Lambda_k,
\]
and \eqref{eq:dirichlet-minmax} yields $\lambda_k<\Lambda_k$.  This gives an elementary proof of the same-index Dirichlet--buckling comparison discussed in \cite{Kelliher2010}, without any assumption on $\partial\Om$.
\end{remark}

\begin{remark}\label{rem:even}
We explain more precisely why the argument does not extend to even $k$. Note first that Theorem~\ref{thm:main} does cover an even index $k$ whenever $\Lambda_k=\Lambda_{k+1}$, since then $\lambda_{k+1}\leq\lambda_{k+2}\leq\Lambda_{k+1}=\Lambda_k$. In general, consider the matrix of $\omega_t$ in the basis \eqref{eq:normalization},
\begin{gather*}
M(t):=\left(\omega_t(u_i,u_j)\right)_{i,j=1}^k=M_0-tM_1,
\\
M_0:=\left(\int_\Om\grad(Tu_i)\cdot\grad u_j\right)_{i,j=1}^k,
\quad
M_1:=\left(\int_\Om(Tu_i)\,u_j\right)_{i,j=1}^k.
\end{gather*}
Since $M(t)$ is skew-symmetric for every $t$, so are $M_0$ and $M_1$, and $p(t):=\det M(t)$ is a polynomial in $t$ of degree at most $k$. The proof of Theorem~\ref{thm:main} requires $p(t)=0$ along a sequence $t\downarrow\Lambda_k$. For odd $k$ this holds because $p\equiv0$. For even $k$ it holds if and only if $p\equiv0$; otherwise $\omega_t$ is non-degenerate for all but finitely many $t$ (at most $k/2$ of them, as $p$ is the square of the Pfaffian of $M(t)$), and the argument gives no information near $\Lambda_k$.

The identity $p\equiv0$ for even $k$ cannot be deduced from the ingredients of the proof. Indeed, let $\Om=(0,\pi)\subset\R$ and replace $T$ by $\frac{\dr}{\dr x}$. Then Lemma~\ref{lem:first-order} and \eqref{eq:buckling-bound}--\eqref{eq:TEnegative} hold with the same (in fact simpler) proofs, and hence so does the proof of Theorem~\ref{thm:main} for odd $k$. On the other hand, $\lambda_j=j^2$, and splitting buckling eigenfunctions into even and odd parts with respect to $x=\pi/2$ gives
\[
\Lambda_{2n-1}=4n^2,
\qquad
\Lambda_{2n}=\frac{4z_n^2}{\pi^2},
\qquad n\geq1,
\]
where $z_n\in\left(n\pi,(n+\tfrac12)\pi\right)$ is the $n$-th positive root of $\tan z=z$; cf.\ \cite[Section~4]{BLPS2021}. Thus $\lambda_{2n}=\Lambda_{2n-1}$, while $\Lambda_{2n}<(2n+1)^2=\lambda_{2n+1}$. Consequently, for $k=2n$ and $\Lambda_k<t<\lambda_{k+1}$, the form $q_t$ has no $(k+1)$-dimensional negative subspace in $H_0^1(0,\pi)$; in particular, $\omega_t$ is non-degenerate for all such $t$, and $p\nequiv0$. A proof of \eqref{eq:payne-conj} at even indices must therefore use properties of planar domains that have no counterpart on the interval. See also Remark \ref{rem:cylinder} for another example illustrating this phenomenon.
\end{remark}

\begin{remark}\label{rem:dimension}
The dimension enters the proof only through Lemma~\ref{lem:first-order}. Let $\Om\subset\R^d$ be a bounded open set, and 
consider a first-order operator $T_a:=a\cdot\grad=\sum_{j=1}^d a_j\partial_j$ with constant coefficients $a=p+\ir q\in\C^d$, $p,q\in\R^d$; the operator $T$ of Lemma~\ref{lem:first-order} is $T_a$ with $d=2$ and $a=(1,\ir)$.
By Plancherel's theorem, \eqref{eq:T1} holds for all $u\in H_0^2(\Om;\C)$ if and only if $|a\cdot\xi|^2=|\xi|^2$ for all $\xi\in\R^d$ (for the ``only if'' part, test with $u(x)=\phi(x)\cos(N\omega\cdot x)$, $\phi\in C_c^\infty(\Om)$, $\omega\in\mathbb{S}^{d-1}$, and let $N\to\infty$), that is, if and only if $pp^{\mathsf T}+qq^{\mathsf T}=I_d$; in this case \eqref{eq:T2} holds as well. Since the matrix on the left has rank at most two, this is impossible for $d\geq3$. For $d=1$ one may take $a=1$, which is the setting of Remark~\ref{rem:even}. For $d=2$ the condition means that $p$ and $q$ are orthonormal, so that $T_a=\mathrm{e}^{\ir\phi}(\partial_1\pm\ir\partial_2)$ for some $\phi\in\R$; thus $T$ is essentially the only admissible choice in the plane.

The case $k=1$ of \eqref{eq:payne-conj} nevertheless holds in every dimension, with $\lambda_k$ and $\Lambda_k$ defined by \eqref{eq:dirichlet-minmax} and \eqref{eq:buckling-minmax}, cf. \cite{DLS26}. Indeed, the proofs of \eqref{eq:buckling-bound} and \eqref{eq:Enegative} do not use $d=2$. Let $u_1$ be a real first buckling eigenfunction. By Plancherel's theorem,
\[
\sum_{j=1}^d\left(\|\grad\partial_ju_1\|_2^2-\Lambda_1\|\partial_ju_1\|_2^2\right)
=\|\Delta u_1\|_2^2-\Lambda_1\|\grad u_1\|_2^2=0,
\]
so $\|\grad\partial_ju_1\|_2^2\leq\Lambda_1\|\partial_ju_1\|_2^2$ for some $j$; also $\partial_ju_1\neq0$  all $j$ since $u_1$ has compact support in $\R^d$ after extension by zero. Since $u_1$ is real,
\[
\int_\Om u_1\,\partial_ju_1=\frac12\int_\Om\partial_j\bigl(u_1^2\bigr)=0,
\qquad
\int_\Om\grad u_1\cdot\grad\partial_ju_1=\frac12\int_\Om\partial_j|\grad u_1|^2=0.
\]
Hence, for $t>\Lambda_1$, the functions $u_1$ and $\partial_ju_1$ are $q_t$-orthogonal and $q_t$-negative, so they span a two-dimensional $q_t$-negative subspace, and \eqref{eq:dirichlet-minmax} gives $\lambda_2<t$, in which we can take the limit $t\downarrow \Lambda_1$. 

For $d\geq3$ and odd $k\geq3$ we know of no substitute for $T$.
\end{remark}

\subsection{The $\bar\partial$ viewpoint and flat surfaces}\label{subsec:geometry}

In the complex coordinate $z=x+\ir y$,
\[
T=\partial_x+\ir\partial_y=2\partial_{\bar z}.
\]
Thus $T$ is the scalar form of the $\bar\partial$-operator. The global parallel orthonormal frame of the Euclidean plane provides a parallel trivialisation of the bundle of $(0,1)$-forms, identifying such forms with complex-valued functions. This gives a geometric explanation for the special role of dimension two, and an immediate extension to flat surfaces with trivial holonomy.

\begin{prop}\label{prop:flat}
Let $(M,g)$ be a compact connected oriented flat Riemannian surface with nonempty smooth boundary and trivial Levi-Civita holonomy. Let $\lambda_k(M,g)$ and $\Lambda_k(M,g)$ denote the corresponding Dirichlet and buckling eigenvalues. Then
\begin{equation}\label{eq:flat}
\lambda_{2m}(M,g)\leq\Lambda_{2m-1}(M,g),
\qquad m\geq1.
\end{equation}
\end{prop}

\begin{proof}
Choose a global parallel oriented orthonormal frame $X,Y$ and set $T=X+\ir Y$. Since the frame is parallel, the identities of Lemma~\ref{lem:first-order} hold verbatim with the Riemannian gradient and Laplacian:
\[
\|Tu\|_2^2=\|\grad_g u\|_2^2,
\qquad
\|\grad_g Tu\|_2^2=\|\Delta_g u\|_2^2,
\qquad
b_{t,g}(Tu,v)=-b_{t,g}(u,Tv).
\]
The proof of Theorem~\ref{thm:main} therefore carries over without change.
\end{proof}

If $(M,g)$ is globally isometric to a planar domain, this is of course just Theorem~\ref{thm:main}. 
The proposition applies, for example, to smoothly bounded subdomains of flat tori, and to flat cylinders, for which it is sharp, see Remark~\ref{rem:cylinder}.
For a flat surface with nontrivial holonomy, and more generally for a curved surface, $\bar\partial u$ is still intrinsically defined but there is no global parallel trivialisation which turns it into a scalar trial function. Thus the $\bar\partial$ interpretation suggests a geometric extension of the method, but does not by itself give the scalar comparison on an arbitrary surface.

\begin{remark}\label{rem:cylinder}
The main idea of the following example is due to  \cite{BL26}. Let $\Om_L:=\mathbb{S}^1\times(0,L)$, with $\mathbb{S}^1=\R/2\pi\mathbb{Z}$, be the flat cylinder of circumference $2\pi$ and height $L$, the boundary conditions being imposed on the two boundary circles. Separating variables, $u=f(y)\mathrm{e}^{\ir nx}$ with $n\in\mathbb{Z}$, the Dirichlet eigenvalues of $\Om_L$ are $n^2+\tfrac{m^2\pi^2}{L^2}$, $m\geq1$, whereas  
the buckling quotient equals 
\[
n^2+\frac{\|f''\|_2^2+n^2\|f'\|_2^2}{\|f'\|_2^2+n^2\|f\|_2^2},
\] 
which by $\|f''\|_2\|f\|_2\geq\|f'\|_2^2$
 is at least $n^2+\tfrac{\|f'\|_2^2}{\|f\|_2^2}>n^2+\tfrac{\pi^2}{L^2}$. The case $n=0$ corresponds to the interval treated in Remark~\ref{rem:even}. Hence, if $L>2\pi\sqrt{J(J+1)}$ for some $J\geq1$ (and therefore $L>2\sqrt{2}\pi$),  then the first $2J+1$ Dirichlet and the first $2J$ buckling eigenvalues of $\Om_L$ all come from $n=0$, and therefore
\[
\lambda_{2m}(\Om_L)=\Lambda_{2m-1}(\Om_L)=\frac{4m^2\pi^2}{L^2},
\qquad 1\leq m\leq J,
\]
\[
\lambda_{2J+1}(\Om_L)-\Lambda_{2J}(\Om_L)=\frac{(2J+1)^2\pi^2-4z_J^2}{L^2}>0,
\]
with $z_J$ defined in Remark~\ref{rem:even}.
Thus \eqref{eq:flat} is an equality at every odd index less than or equal to $2J-1$, while its even-index analogue fails for every even index less than or equal to $2J$. The point is that the whole mechanism of \S\ref{subsec:planar-proof} is available on $\Om_L$: the operator $T$ is globally defined and Lemma~\ref{lem:first-order} holds verbatim. This sharpens the conclusion of Remark~\ref{rem:even}: a proof of \eqref{eq:payne-conj} at even indices must use a property distinguishing planar domains from flat cylinders.
\end{remark}

As was mentioned in the introduction, there is a link between Payne's conjecture and comparison inequalities for Dirichlet and Neumann eigenvalues, see \cite{Payne1955, Fri91, Fil04}
as well as \cite[\S 3.2.4]{LMP} for earlier results and background material.
Our  $\bar\partial$ viewpoint is closely related to the differential-form interpretation of Rohleder's approach to Neumann--Dirichlet inequalities \cite{Rohleder2025} which was developed in   \cite{FriesGoffengMiranda2026, Muravyev2026, HuaMunchZhang2026}. 

\begin{remark}[Friedrichs and Kre\u{\i}n extensions]\label{rem:krein}
From the viewpoint of extension theory, the Dirichlet Laplacian is the Friedrichs extension of the minimal Laplacian, whereas the positive eigenvalues of its Kre\u{\i}n--von Neumann extension are precisely the buckling eigenvalues; see \cite{AshbaughEtAl2010}. Thus Payne's conjecture may also be viewed as a comparison between the Friedrichs and Kre\u{\i}n extensions, analogous in spirit to Dirichlet--Neumann spectral comparisons. 
\end{remark}

\subsection*{Acknowledgements}

ML has been partially supported by an INI Simons Fellowship. IP has been partially supported by NSERC and FRQNT. Both authors
have been partially supported by the AI for Math Fund, an initiative by Renaissance Philanthropy. 

The authors would like to thank Davide Buoso for useful discussions and for letting them know about the upcoming preprint \cite{BL26} which motivated Remark \ref{rem:cylinder}.


\subsection*{Data availability statement}

There are no additional data.

\subsection*{AI usage disclosure}
The idea of the proof of Theorem \ref{thm:main} came up in an exchange with  ChatGPT Pro 5.6 in response to a query on whether the argument in \cite{TangZhang2026} can be adapted to Payne's conjecture. The authors also used a workflow combining ChatGPT Pro 5.6 and Claude Opus 5 for 
proof verification, literature searches, and drafting and editing parts of the text.

All statements and proofs were checked by the authors, who take full responsibility for the content.

\end{document}